\documentclass[3p, twocolumn]{elsarticle}

\usepackage{amssymb}
\usepackage{amsmath}
 \usepackage{amsthm}
 \usepackage{complexity}
 \usepackage{tikz}
 \usetikzlibrary{calc}
 \usepackage{mdframed}
 \usepackage{comment}
 \newtheorem{theorem}{Theorem}[section]
\newtheorem{claim}[theorem]{Claim}
\newtheorem{lemma}[theorem]{Lemma}

\newtheorem{observation}[theorem]{Observation}

\def\ve#1{\mathchoice{\mbox{\boldmath$\displaystyle\bf#1$}}
{\mbox{\boldmath$\textstyle\bf#1$}}
{\mbox{\boldmath$\scriptstyle\bf#1$}}
{\mbox{\boldmath$\scriptscriptstyle\bf#1$}}}

\usepackage[tamon,krishna]{optional} 

\journal{Operations Research Letters}

\begin{document}

\begin{frontmatter}

\title{A note on the exact partition polytope of Frieze and Teng}

\author{Krishna Narayanan, Tamon Stephen} 

\affiliation{organization={Department of Mathematics, Simon Fraser University},
            city={Burnaby},
            postcode={V5A 1S6}, 
            state={British Columbia},
            country={Canada}}

\begin{abstract}
In 1994, Frieze and Teng~\cite{FT94} proposed an integer linear programming formulation of the $\NP$-Complete \textsc{Exact Partition} problem, whose LP-relaxation they claimed was non-degenerate. Contrary to their claim, we show how an instance of \textsc{Exact Partition} can produce a degenerate polytope, and study conditions for which this can happen. We then give details of one of the smallest such degenerate Frieze-Teng polytopes, along with a closely related non-degenerate Frieze-Teng polytope that encodes an equivalent problem. We note that for the purposes of the complexity results in~\cite{FT94} and \cite{NS26}, these degenerate polytopes can be avoided via a simple preprocessing step. 

\end{abstract}



\begin{keyword}
Exact Partition \sep Polytope \sep Degeneracy 



\end{keyword}

\end{frontmatter}



\section{Introduction}
\label{sec1}

\textsc{Exact Partition} is a variant of the well-known $\NP$-Complete \textsc{Partition} problem, which was one of Karp's 21 $\NP$-Complete problems \cite{Karp72}. \textsc{Partition} asks if in a given set of integers, we can obtain two subsets of equal sum. \textsc{Exact Partition} additionally requires the subsets to be of equal size. 

In \cite{FT94}, Frieze and Teng modified a version of an integer linear programming formulation of \textsc{Exact Partition} proposed by Korte and Schrader in \cite{KorteSchrader}. The polytope arising from the LP-relaxation of Frieze and Teng's formulation is used in the context of studying the computational complexity of diameters and monotone diameters of polytopes. 
Recent activity highlighting complexity questions on simple polytopes includes results by the authors of this note~\cite{NS26} and by Black and Steiner~\cite{black2026findingshortpathssimple}.

Frieze and Teng claimed that their polytope is non-degenerate (Proposition 2, \cite{FT94}). We identify an issue in their proof, but note that the problematic instances can be sidestepped in the derivation of complexity results.  
We describe a minimal counterexample of \textsc{Exact Partition} with four elements that gives a degenerate Frieze-Teng polytope, and compare it to an equivalent \textsc{Exact Partition} instance whose Frieze-Teng polytope is non-degenerate. We finish with a discussion of conditions that can cause degeneracy in this polytope. 
\vspace{-1em}

\section{Preliminaries and Frieze and Teng's formulation}
We begin with some definitions. An \textit{$\mathcal{H}$-polytope} (henceforth, simply a \textit{polytope}) is the intersection of finitely many closed halfspaces in $\mathbb{R}^d$ that is \textit{bounded}, i.e., it contains no ray $\{\ve{x}+t\ve{y}: t \geq 0\}$ for $\ve{x} \in \mathbb{R}^d$, $\ve{y} \neq \ve{0}$. The \textit{dimension} of a polytope is the dimension of its affine hull. A $d$-dimensional polytope is \textit{simple} if every vertex is adjacent to exactly $d$ edges in the 1-skeleton, or equivalently, contained in exactly $d$ facets. When a vertex is at the intersection of more than $d$ facets, we call it a \textit{degenerate} vertex.  If a polytope contains any degenerate vertices, then we call it a \textit{degenerate polytope}.

The \textsc{Exact Partition} problem is defined as follows.

\paragraph{Problem: \textsc{Exact Partition}}
\begin{itemize}
\item[] \textit{Input:} A finite set \( \mathcal{A} = \{s_1, s_2, \ldots, s_{2m}\} \) of integers.
\item[] \textit{Question:} Does there exist a subset \( \mathcal{A'} \subset \mathcal{A} \) with \( |\mathcal{A'}| = m \) such that
\[\sum_{s \in \mathcal{A'}} s = \sum_{t \in \mathcal{A} \setminus \mathcal{A'}} t ?\]
\end{itemize}

Korte and Schrader \cite{KorteSchrader} studied this problem via a $0/1$ linear programming formulation with two knapsack-like constraints.  The first limits the number of elements in the set to $m$ and the second limits the sum of the selected elements to $\displaystyle \frac{S}{2}$ where $S := \displaystyle \sum_{i=1}^{2m} s_i$.
Frieze and Teng~\cite{FT94} noted that the polytope arising from this formulation, which they call \textit{ILP1}, is highly degenerate. They provide another formulation, \textit{ILP2}, which they claim is non-degenerate.

Let $s_{\max}$ be the largest element in a given problem instance $\mathcal{A}$. Take $M := S + 1$, $d_i := s_{\max}-s_i$ and $\epsilon := \frac{1}{2M}$.  The Frieze-Teng formulation of \textsc{Exact Partition} is:
\vspace{1em}

\textbf{(ILP2) }
\begin{align*}
& \text{Maximize } & \sum_{i=1}^{2m} x_i \text{ with } x_i \in \{0, 1\} \qquad \qquad \qquad \\
& \text{subject to } & \sum_{i=1}^{2m} (M + s_i)x_i \leq \tfrac{1}{2} S + mM + \epsilon \qquad \tag{K1} \\
& & \sum_{i=1}^{2m} (M + d_i)x_i \leq \tfrac{1}{2} \sum_{i=1}^{2m} d_i + mM + \epsilon \tag{K2}
\end{align*}

Observe that (K1) limits the sum of elements in the chosen subset, while (K2) limits the number of selected elements. We assume that the elements of $\mathcal{A}$ are positive.
If this does not hold we can translate the problem to this scenario by adding an appropriate constant to each element of $\mathcal{A}$.

Let us denote the polytope arising from the linear programming relaxation of ILP2 by $P_R$. Then $P_R$ is a polytope obtained from the unit $2m$-cube $0 \le x_i \le 1$ by cutting it with two knapsack-type constraints (K1) and (K2) with positive co-efficients. We will refer to the constraints $0 \le x_i \le 1$ as \textit{box constraints}, and to (K1) and (K2) as the \textit{knapsack constraints}.

\subsection{Some characteristics of $P_R$ and related notation}\label{vertex}
For a given vertex $\ve{v} \in P_R$, we use the following notation in the remainder of this document.
\begin{itemize}
    \item $I_1(\ve{v}) = \{i \in [2m] \colon x_i=1\}$
    \item $I_0(\ve{v}) = \{i \in [2m] \colon x_i=0\} $
    \item $I_f(\ve{v}) = \{i \in [2m] \colon 0<x_i<1\} $
\end{itemize}
Additionally, for a vertex $\ve{v}$, let us also denote $I^*(\ve{v}) = I_1(\ve{v}) \cup I_f(\ve{v})$.

Frieze and Teng claim that $P_R$ is a simple polytope. 

\begin{claim}[\cite{FT94}]\label{nondegLPre}
    When \(\epsilon = \frac{1}{2M}\), the polytope defined by the linear programming relaxation of ILP2 is simple (non-degenerate).
\end{claim}
The above claim is equivalent to saying that only $2m$ of the constraints are active at any vertex of $P_R$. 

Following Lemma 5 of~\cite{FT94}, the vertices of $P_R$ can be partitioned into three classes:
\begin{itemize}
    \item[] $V_0$ - Vertices lying on $2m$ box constraints. Any such point is integer, and does not satisfy (K1) or (K2) due the choice of $\epsilon$. 
    \item[] $V_1$ - Vertices lying on exactly $2m-1$ box constraints.  These will lie on at least one of (K1) or (K2). 
    \item[] $V_2$ - Vertices lying on exactly $2m-2$ box constraints.  These will lie on both (K1) and (K2).
\end{itemize}

We now outline Frieze and Teng's proof of Claim \ref{nondegLPre} and identify an issue.
\begin{proof}
Assume that $P_R$ is degenerate, which implies that there is a vertex $\ve{x}$ in $P_R$ which lies on $2m+1$ active constraints.  

If $\ve{x} \in V_0$, then $x_i \in \{0,1\} for \quad i=1,\ldots, 2m$.  The choice of $\epsilon$ in the definitions of (K1) and (K2) is set so that they do not contain any integer points and $\ve{x}$ lies on exactly $2m$ constraints.
If $\ve{x} \in V_2$, since there are only two remaining constraints, $\ve{x}$ lies on at most $2m$ constraints.

Now suppose $\ve{x} \in V_1$.  Then $|\{x_i \in \{0,1\}\}| = 2m-1$. Without loss of generality, assume that $0< x_1<1 $. For $\ve{x}$ to be degenerate, both (K1) and (K2) need to be active. 

\begin{align*}
    \sum_{i=1}^{2m} (M + s_i) x_i &= \frac{1}{2}S + mM + \epsilon \tag{K1}\\ 
\sum_{i=1}^{2m} (M + d_i)x_i &= \frac{1}{2} \sum_{i=1}^{2m} d_i + mM + \epsilon \tag{K2}
\end{align*}

Frieze and Teng claim that after eliminating $x_1$, they obtain an integer $c$ such that $$\epsilon = \frac{c}{2d_1}$$ which contradicts the assumption that $\epsilon = \frac{1}{2M}$. However, we observe that this expression obtained for $\epsilon$ is not accurate. In particular, when we isolate $x_1$, we obtain $x_1 = \frac{c_1}{(s_{\max}-2s_1)}$, where $c_1 \in \mathbb{Z}$ and by substituting $x_1$, we obtain $\epsilon = \frac{c_3}{2(d_1 - s_1)}$, where $c_3 \in \mathbb{Z}$. We detail this below. 

Recall that by our assumption, both (K1) and (K2) are active at $\ve{x}$. We follow Frieze and Teng to eliminate $x_1$ by isolating it. By subtracting (K1) from (K2), we have 
$$\sum_{i=1}^{2m} (d_i - s_i)x_i = \frac{1}{2}\sum_{i=1}^{2m} d_i - \frac{1}{2}S$$

Substituting $d_i = s_{\max} - s_i$, we have
\begin{align*}
\sum_{i=1}^{2m} ((s_{\max} - s_i) - s_i)x_i &= \frac{1}{2}\sum_{i=1}^{2m} (s_{\max} - s_i) - \frac{1}{2}S \\
\sum_{i=1}^{2m} (s_{\max} - 2s_i)x_i &= \frac{1}{2}\left(2m\cdot s_{\max} - S\right) - \frac{1}{2}S \\
\end{align*}

If $s_{\max} = 2s_1$ then $x_1$ drops out of the equation.  When $s_{\max} \ne 2s_1$, we can proceed to isolate $x_1$
\begin{align*}
  (s_{\max} - 2s_1)x_1 + \sum_{i=2}^{2m} (s_{\max} - 2s_i)x_i&= m\cdot s_{\max} - S
\end{align*}
\begin{align*}
  x_1 &= \frac{m\cdot s_{\max} - S - \sum_{i=2}^{2m} (s_{\max} - 2s_i)x_i}{(s_{\max} - 2s_1)}
\end{align*}

Since $x_i \in \{0,1\}$ for all $i\in\{2,...,2m\}$, $m s_{\max} - S - \sum_{i=2}^{2m} (s_{\max} - 2s_i)x_i$ in the numerator is an integer. Let this integer be $c_1$. Then, we have

\begin{align*}
    x_1 &= \frac{c_1}{(s_{\max} - 2s_1)}
\end{align*}

Recalling $0<x_1<1$, we substitute $x_1$ into (K1) to isolate $\epsilon$.

\begin{align*}
    \epsilon &= \sum_{i=1}^{2m}(M+s_i)x_i - \frac{1}{2}S - mM \\
    \epsilon &= (M+s_1)x_1 + \sum_{i=2}^{2m}(M+s_i) - \frac{1}{2}S - mM\\
    \epsilon &= (M+s_1)\frac{c_1}{(s_{\max} - 2s_1)} + \sum_{i=2}^{2m}(M+s_i) - \frac{1}{2}S - mM 
\end{align*}

Once more, since $x_i \in \{0,1\}$ for all $i\in\{2,...,2m\}$, we have $\displaystyle c_2 = \sum_{i=2}^{2m} (M + s_i)x_i - mM$ such that $c_2 \in \mathbb{Z}$. Then, we have
\begin{align*}
    \epsilon &= (M+s_1)\frac{c_1}{(s_{\max} - 2s_1)} + c_2 - \frac{1}{2}S
\end{align*}
Since $d_1 = s_{\max} - s_1$, we have  $s_{\max} - 2s_1 = d_1 - s_1$. We then have 
\begin{align*}
    \epsilon &= (M+s_1)\frac{c_1}{(d_1 - s_1)} + c_2 - \frac{1}{2}S\\
    \epsilon &= \frac{2((M+s_1)c_1+c_2(d_1 - s_1))-S(d_1 - s_1)}{2(d_1 - s_1)}
\end{align*}

Clearly, $c_3 = 2((M+s_1)c_1+c_2(d_1 - s_1))-S(d_1 - s_1)$ is an integer. Therefore, we have 
\begin{align*}
   \epsilon &= \frac{c_3}{2(d_1 - s_1)}
\end{align*}

which contradicts the assumption that $\epsilon = \frac{1}{2M}$ and the proof follows.

\end{proof}

We adjust Frieze and Teng's lemma to remove the offending case:

\begin{lemma}\label{new}
    When \(\epsilon = \frac{1}{2M}\) and there is no $i \in \{1,\ldots,2m\}$ such that $s_i = \frac{s_{\max}}{2}$, the polytope defined by the linear programming relaxation of ILP2 is simple.
\end{lemma}

When the conditions for the above Lemma are satisfied, we have the following alternate characterization of $V_0$, $V_1$ and $V_2$ respectively. This is implied in~\cite{FT94} and described in~\cite{KN25} and~\cite{NS26}.

\begin{itemize}
    \item $V_0$ - Vertices that lie only on the box constraints and not on either of (K1) or (K2).
    \item $V_1$ - Vertices that lie on either of (K1) or (K2), but not both, as these vertices lie on exactly $2m-1$ box constraints and a total of $2m$ constraints. 
    \item $V_2$ - Vertices that lie on both (K1) and (K2). The simplicity of $P_R$ implies that there are exactly two fractional components by a similar argument as in $V_1$. 
\end{itemize}

For the case where $s_i = \frac{s_{\max}}{2}$, it turns out we \emph{can} get degenerate polytopes. In particular, we obtain polytopes where vertices with one fractional variable lie on both (K1) and (K2), which violates the above implication of simplicity. In Section~\ref{se:degen_example} we study one such polytope and generalize the conditions under which degenerate vertices occur.  

\subsection{Preprocessing to avoid degenerate polytopes in complexity reductions}

Frieze and Teng introduce this lemma to prove that computing a vertex's \emph{eccentricity} (or \emph{radius}), the maximum shortest pivot distance to any other vertex on a polytope, is $\NP$-hard via a reduction from \textsc{Exact Partition}.  Similarly, the authors of this note use this in the context of \emph{monotone} eccentricity~\cite{NS26}. 

We remark that this still holds, as the degenerate polytopes that can occur in this reduction can be avoided by simple pre-processing.
Recall from Section 2 that we assumed that all input integers in the \textsc{Exact Partition} instance are non-negative, and if not, can be translated accordingly by adding an appropriate constant to every element of the instance. 

A degenerate vertex requires an even $s_{\max}$ in the \textsc{Exact Partition} instance. To prevent this, adding one to each element ensures the new $s_{\max}$ is odd, making $s_i = s_{\max}/2$ impossible while preserving its status as the largest element. By Lemma~\ref{new}, this transformation guarantees the resulting exact partition polytope contains no degenerate vertices.

\section{Degeneracy in $P_R$} \label{se:degen_example}

\subsection{An exact partition problem that gives rise to a degenerate $P_R$}\label{counter}
Consider the following instance $A = \{3,3,4,2\}$ of \textsc{Exact Partition}. In this case, $2m = 4$, $m=2$ and $s_{\max}=4$. In this instance, we notice that there is an exact partition of subsets, $\{3,3\}$ and $\{4,2\}$. Notice also that $s_{max}=s_3=4$, and that $s_4=\frac{s_{\max}}{2}$. 

The exact partition polytope arising from this instance has a degenerate vertex $\ve{x}=(1, 1, 0, \frac{1}{390})$.

We see that $\ve{x}$ lies on three box constraints $x_1,x_2\leq 1$ and $x_3\geq0$, along with both knapsack-like constraints (K1) and (K2).

Figure \ref{fig:p1_full} illustrates a Schlegel diagram of $P_R$ obtained from the instance $A = \{3,3,4,2\}$.  This was drawn using the tikzpackage on polymake \cite{polymake} and modified to show a clearer description of the clustered vertices. 

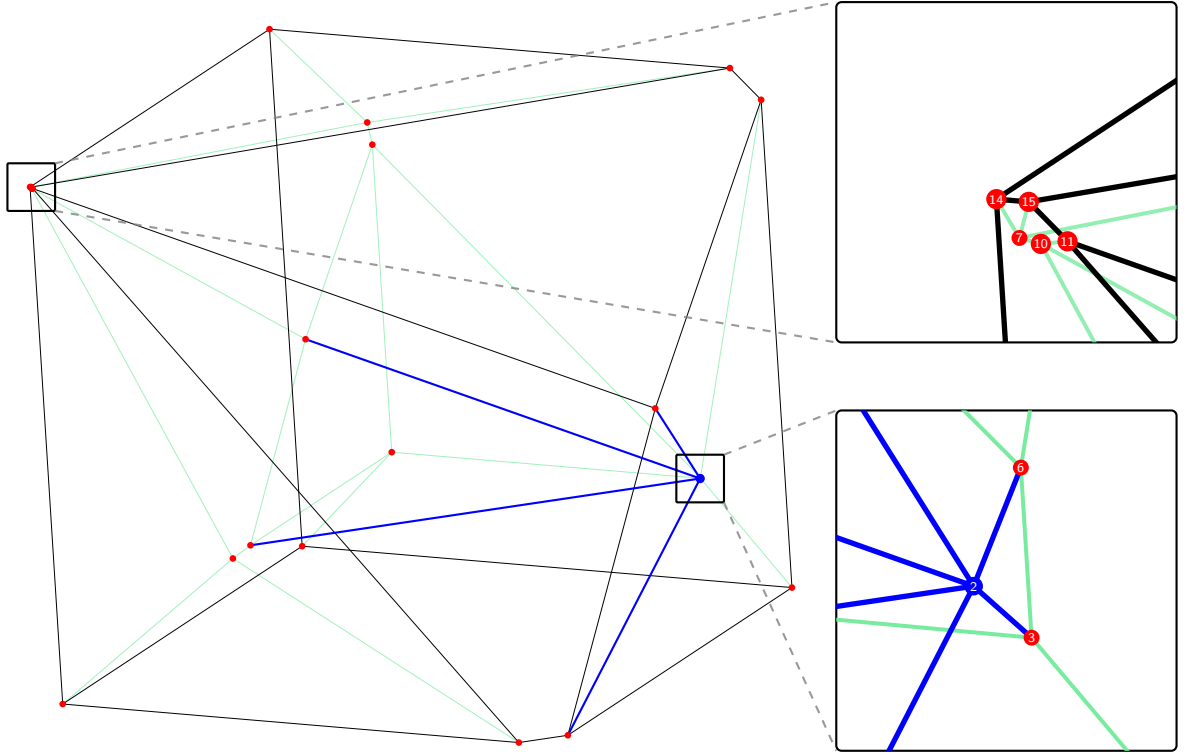
\begin{figure*}[p]
    \centering
    \vfill
 \begin{tikzpicture}[scale=0.9]

  \definecolor{edgecolor_p1_0}{rgb}{ 0 0 0 }
  \definecolor{edgecolor_p1_1}{rgb}{ 0.4667 0.9255 0.6196 }
  \definecolor{edgecolor_p1_2}{rgb}{ 0 0 1 } 

  \tikzstyle{v_norm_base}      = [circle, fill=red, inner sep=0pt, minimum size=2.5pt]
  \tikzstyle{v_highlight_base} = [circle, fill=blue, inner sep=0pt, minimum size=3.5pt]
  \tikzstyle{v_norm_z}         = [circle, fill=red, text=white, inner sep=0.5pt, minimum size=6pt, font=\sffamily\tiny]
  \tikzstyle{v_highlight}      = [circle, fill=blue, text=white, inner sep=0.5pt, minimum size=7pt, font=\sffamily\tiny]

  \tikzstyle{e_inner_base} = [very thin, color=edgecolor_p1_1, opacity=0.7]
  \tikzstyle{e_outer_base} = [very thin, color=edgecolor_p1_0]
  \tikzstyle{e_deg_base}   = [thick, color=edgecolor_p1_2]
  
  \tikzstyle{e_cluster_z}  = [line width=1.5pt, color=edgecolor_p1_1, opacity=1] 
  \tikzstyle{e_deg_z}      = [line width=2pt, color=edgecolor_p1_2]

  \begin{scope}[xshift=-2cm, x={(0.9cm,-0.076cm)}, y={(-0.06cm,0.95cm)}, z={(-0.44cm,-0.29cm)}, scale=8]
    
    \coordinate (v0) at (0.350656, 0.302076, 0.892073);
    \coordinate (v1) at (0.350656, 0.617323, 0.618859);
    \coordinate (v2) at (0.981151, 0.302076, 0.303611);
    \coordinate (v3) at (0.981151, 0.302076, 0.301994);
    \coordinate (v4) at (0.335076, 0.288654, 0.935489);
    \coordinate (v5) at (0.350656, 0.302076, 0.301994);
    \coordinate (v6) at (0.981151, 0.303502, 0.301994);
    \coordinate (v7) at (0.00133504, 0.999743, 0.999743);
    \coordinate (v8) at (0.337067, 0.935184, 0.290292);
    \coordinate (v9) at (0.350656, 0.89691, 0.301994);
    \coordinate (v10) at (0.00284355, 0.996897, 0.999453);
    \coordinate (v11) at (0.00512821, 0.997436, 1);
    \coordinate (v12) at (0, 1, 0);
    \coordinate (v13) at (0.939904, 1, 0);
    \coordinate (v14) at (0, 1, 1);
    \coordinate (v15) at (0.00240385, 1, 1);
    \coordinate (v16) at (1, 0.5, 0.502564);
    \coordinate (v17) at (1, 0, 0.935897);
    \coordinate (v18) at (1, 0, 0);
    \coordinate (v19) at (0.931319, 0, 1);
    \coordinate (v20) at (0, 0, 1);
    \coordinate (v21) at (0, 0, 0);
    \coordinate (v22) at (1, 0.943439, 0);

    \foreach \i/\k in {1/0,4/0,5/0,5/3,6/3,8/7,9/1,9/5,9/6,9/8,10/1,10/4,10/7,11/10,12/8,13/8,14/7,15/7,18/3,19/4,20/4,21/5,22/6} {
      \draw[e_inner_base] (v\i) -- (v\k);
    }
    \foreach \i/\k in {13/12,14/12,15/11,15/13,15/14,16/11,17/16,18/17,19/11,19/17,20/14,20/19,21/12,21/18,21/20,22/13,22/16,22/18} {
      \draw[e_outer_base] (v\i) -- (v\k);
    }

    \foreach \i/\k in {2/0,2/1,3/2,6/2,16/2,17/2} {
      \draw[e_deg_base] (v\i) -- (v\k);
    }
    
    \foreach \i in {0,1,3,4,5,6,7,8,9,10,11,12,13,14,15,16,17,18,19,20,21,22} { \node at (v\i) [v_norm_base] {}; }
    \node at (v2) [v_highlight_base] {};

    \coordinate (SE_target) at (v2);
    \coordinate (NW_target) at (v15);
  \end{scope}
  
  \draw[thick, black, rounded corners=0.5pt] ($(SE_target) - (0.35cm, 0.35cm)$) rectangle ($(SE_target) + (0.35cm, 0.35cm)$);
  \draw[thick, black, rounded corners=0.5pt] ($(NW_target) - (0.35cm, 0.35cm)$) rectangle ($(NW_target) + (0.35cm, 0.35cm)$);

  \coordinate (NW_ZoomPos) at ($(SE_target) + (4.5cm, 4.5cm)$); 
  \draw[thick, black, fill=white, rounded corners=2pt] ($(NW_ZoomPos) - (2.5cm, 2.5cm)$) rectangle ($(NW_ZoomPos) + (2.5cm, 2.5cm)$);
  
  \begin{scope}[shift={(NW_ZoomPos)}]
    \clip[rounded corners=2pt] (-2.5cm, -2.5cm) rectangle (2.5cm, 2.5cm);
    
    \begin{pgfinterruptboundingbox}
      
      \begin{scope}[x={(0.9cm,-0.076cm)}, y={(-0.06cm,0.95cm)}, z={(-0.44cm,-0.29cm)}, scale=220]
        
        \begin{scope}[shift={(-0.002404, -1.000000, -1.000000)}] 
          
          \foreach \i/\x/\y/\z in {0/0.350656/0.302076/0.892073, 1/0.350656/0.617323/0.618859, 2/0.981151/0.302076/0.303611, 3/0.981151/0.302076/0.301994, 4/0.335076/0.288654/0.935489, 5/0.350656/0.302076/0.301994, 6/0.981151/0.303502/0.301994, 7/0.00133504/0.999743/0.999743, 8/0.337067/0.935184/0.290292, 9/0.350656/0.89691/0.301994, 10/0.00284355/0.996897/0.999453, 11/0.00512821/0.997436/1, 12/0/1/0, 13/0.939904/1/0, 14/0/1/1, 15/0.00240385/1/1, 16/1/0.5/0.502564, 17/1/0/0.935897, 18/1/0/0, 19/0.931319/0/1, 20/0/0/1, 21/0/0/0, 22/1/0.943439/0} {
              \coordinate (nw\i) at (\x, \y, \z);
          }
          
          \coordinate (nw7)  at (-0.0005, 0.9960, 0.9960);
          
          \foreach \i/\k in {1/0,4/0,5/0,5/3,6/3,8/7,9/1,9/5,9/6,9/8,10/1,10/4,10/7,11/10,12/8,13/8,14/7,15/7,18/3,19/4,20/4,21/5,22/6,2/0,2/1,3/2,6/2,16/2,17/2} {
            \draw[line width=1.5pt, color=edgecolor_p1_1, opacity=0.8] (nw\i) -- (nw\k);
          }
          \foreach \i/\k in {13/12,14/12,15/11,15/13,15/14,16/11,17/16,18/17,19/11,19/17,20/14,20/19,21/12,21/18,21/20,22/13,22/16,22/18} {
            \draw[line width=2pt, color=edgecolor_p1_0] (nw\i) -- (nw\k);
          }

          \foreach \i in {7, 10, 11, 14, 15} { \node at (nw\i) [v_norm_z] {\i}; }
          
        \end{scope}
      \end{scope}
    \end{pgfinterruptboundingbox}
  \end{scope}

  \coordinate (SE_ZoomPos) at ($(SE_target) + (4.5cm, -1.5cm)$); 
  \draw[thick, black, fill=white, rounded corners=2pt] ($(SE_ZoomPos) - (2.5cm, 2.5cm)$) rectangle ($(SE_ZoomPos) + (2.5cm, 2.5cm)$);
  
  \begin{scope}[shift={(SE_ZoomPos)}]
    \clip[rounded corners=2pt] (-2.5cm, -2.5cm) rectangle (2.5cm, 2.5cm);
    
    \begin{pgfinterruptboundingbox}
      \begin{scope}[x={(0.9cm,-0.076cm)}, y={(-0.06cm,0.95cm)}, z={(-0.44cm,-0.29cm)}, scale=220]
        \begin{scope}[shift={(-0.981151, -0.302076, -0.306000)}] 
          
          \foreach \i/\x/\y/\z in {0/0.350656/0.302076/0.892073, 1/0.350656/0.617323/0.618859, 2/0.981151/0.302076/0.303611, 3/0.981151/0.302076/0.301994, 4/0.335076/0.288654/0.935489, 5/0.350656/0.302076/0.301994, 6/0.981151/0.303502/0.301994, 7/0.00133504/0.999743/0.999743, 8/0.337067/0.935184/0.290292, 9/0.350656/0.89691/0.301994, 10/0.00284355/0.996897/0.999453, 11/0.00512821/0.997436/1, 12/0/1/0, 13/0.939904/1/0, 14/0/1/1, 15/0.00240385/1/1, 16/1/0.5/0.502564, 17/1/0/0.935897, 18/1/0/0, 19/0.931319/0/1, 20/0/0/1, 21/0/0/0, 22/1/0.943439/0} {
              \coordinate (se\i) at (\x, \y, \z);
          }

          \coordinate (se2) at (0.981151, 0.302076, 0.306000); 
          \coordinate (se3) at (0.981151, 0.296000, 0.298000); 
          \coordinate (se6) at (0.981151, 0.308000, 0.298000); 

          \foreach \i/\k in {1/0,4/0,5/0,8/7,9/1,9/5,9/6,9/8,10/1,10/4,10/7,11/10,12/8,13/8,14/7,15/7,19/4,20/4,21/5} {
            \draw[line width=1pt, color=edgecolor_p1_1, opacity=0.5] (se\i) -- (se\k);
          }
          \foreach \i/\k in {13/12,14/12,15/11,15/13,15/14,16/11,17/16,18/17,19/11,19/17,20/14,20/19,21/12,21/18,21/20,22/13,22/16,22/18} {
            \draw[line width=1pt, color=edgecolor_p1_0, opacity=0.5] (se\i) -- (se\k);
          }

          \foreach \i/\k in {5/3,6/3,9/6,18/3,22/6} {
            \draw[e_cluster_z] (se\i) -- (se\k);
          }

          \foreach \i/\k in {2/0,2/1,3/2,6/2,16/2,17/2} {
            \draw[e_deg_z] (se\i) -- (se\k);
          }

          \foreach \i in {0,1,16,17} { \node at (se\i) [v_norm_z] {\i}; }
          \node at (se3) [v_norm_z] {3};
          \node at (se6) [v_norm_z] {6};
          \node at (se2) [v_highlight] {2};
          
        \end{scope}
      \end{scope}
    \end{pgfinterruptboundingbox}
  \end{scope}

  \draw[thick, dashed, gray, opacity=0.8] ($(NW_target) + (0.35cm, 0.35cm)$) -- ($(NW_ZoomPos) + (-2.5cm, 2.5cm)$);
  \draw[thick, dashed, gray, opacity=0.8] ($(NW_target) + (0.35cm, -0.35cm)$) -- ($(NW_ZoomPos) + (-2.5cm, -2.5cm)$);

  \draw[thick, dashed, gray, opacity=0.8] ($(SE_target) + (0.35cm, 0.35cm)$) -- ($(SE_ZoomPos) + (-2.5cm, 2.5cm)$);
  \draw[thick, dashed, gray, opacity=0.8] ($(SE_target) + (0.35cm, -0.35cm)$) -- ($(SE_ZoomPos) + (-2.5cm, -2.5cm)$);

\end{tikzpicture}

    \caption{Schlegel diagram of $P_R$ for the $\{3,3,4,2\}$ instance on the outer facet $x_1=0$ (facet edges in black).}
    \label{fig:p1_full}
    \vfill
\end{figure*}

In the inset image of one of the clusters of vertices, note the existence of a vertex labeled 2 in blue, which has an adjacency of six and not four, as Frieze and Teng's claim would have implied.

To ensure that the resulting exact partition polytope from the instance $A = \{3,3,4,2\}$ has no degenerate vertices, we subtract 1 from every element of this instance to obtain  $\{2,2,3,1\}$. As explained, $s_{\max}$'s index from the original instance remains unchanged. Since it is an odd number, there are no elements which can take the value $\frac{s_{\max}}{2}$, thus ensuring that all vertices are non-degenerate. The resulting polytope's Schlegel diagram is depicted in Figure \ref{fig:p2_full}. 

\begin{figure*}[p]
    \raggedleft
    \vfill
 \begin{tikzpicture}[scale=0.9]

  \definecolor{edgecolor_p2_0}{rgb}{ 0 0 0 }
  \definecolor{edgecolor_p2_1}{rgb}{ 0.4667 0.9255 0.6196 }

  \tikzstyle{v_norm_base}  = [circle, fill=red, inner sep=0pt, minimum size=2.5pt]
  \tikzstyle{v_norm_z}     = [circle, fill=red, text=white, inner sep=0.5pt, minimum size=7pt, font=\sffamily\tiny]

  \tikzstyle{e_inner_base} = [very thin, color=edgecolor_p2_1, opacity=0.7]
  \tikzstyle{e_outer_base} = [very thin, color=edgecolor_p2_0]
  
  \tikzstyle{e_inner_z}    = [line width=1.2pt, color=edgecolor_p2_1, opacity=0.85] 
  \tikzstyle{e_outer_z}    = [line width=1.5pt, color=edgecolor_p2_0]

  \def\loadCoordinates{
    \coordinate (v0)  at (0.284202, 0.311231, 0.31097);
    \coordinate (v1)  at (0.267292, 0.985725, 0.292467);
    \coordinate (v2)  at (0.265547, 0.290802, 0.985575);
    \coordinate (v3)  at (0.957793, 0.311231, 0.31097);
    \coordinate (v4)  at (0.00212571, 0.999886, 0.999885);
    \coordinate (v5)  at (0.00333575, 0.997186, 0.999819);
    \coordinate (v6)  at (0.957793, 0.314349, 0.31097);
    \coordinate (v7)  at (0.957793, 0.312122, 0.313643);
    \coordinate (v8)  at (0.957793, 0.311231, 0.314372);
    \coordinate (v9)  at (0.284202, 0.311231, 0.926728);
    \coordinate (v10) at (0.284202, 0.648917, 0.650439);
    \coordinate (v11) at (0.284202, 0.931808, 0.31097);
    \coordinate (v12) at (1, 0, 0.914141);
    \coordinate (v13) at (1, 0.501323, 0.503968);
    \coordinate (v14) at (1, 0.921296, 0);
    \coordinate (v15) at (0.00793651, 0.997354, 1);
    \coordinate (v16) at (0, 1, 1);
    \coordinate (v17) at (0, 0, 0);
    \coordinate (v18) at (0.914141, 1, 0);
    \coordinate (v19) at (0, 1, 0);
    \coordinate (v20) at (0.905556, 0, 1);
    \coordinate (v21) at (0, 0, 1);
    \coordinate (v22) at (1, 0, 0);
    \coordinate (v23) at (0.00505051, 1, 1);
  }

  \def\drawInnerEdges#1{
    \foreach \i/\k in {3/0, 4/1, 5/2, 5/4, 6/3, 7/6, 8/3, 8/7, 9/0, 9/2, 9/8, 10/5, 10/7, 10/9, 11/0, 11/1, 11/6, 11/10, 12/8, 13/7, 14/6, 15/5, 16/4, 17/0, 18/1, 19/1, 20/2, 21/2, 22/3, 23/4} {
      \draw[#1] (v\i) -- (v\k);
    }
  }

  \def\drawOuterEdges#1{
    \foreach \i/\k in {13/12, 14/13, 15/13, 18/14, 19/16, 19/17, 19/18, 20/12, 20/15, 21/16, 21/17, 21/20, 22/12, 22/14, 22/17, 23/15, 23/16, 23/18} {
      \draw[#1] (v\i) -- (v\k);
    }
  }

 
  \begin{scope}[x={(0.9cm,-0.076cm)}, y={(-0.06cm,0.95cm)}, z={(-0.44cm,-0.29cm)}, scale=8]
    
    \loadCoordinates
    \drawInnerEdges{e_inner_base}
    \drawOuterEdges{e_outer_base}
    \foreach \i in {0,...,23} { \node at (v\i) [v_norm_base] {}; }

    \coordinate (SE_target) at (v3);
    \coordinate (NW_target) at (v16);
  \end{scope}
  
  \draw[thick, black, rounded corners=0.5pt] ($(SE_target) - (0.35cm, 0.35cm)$) rectangle ($(SE_target) + (0.35cm, 0.35cm)$);
  \draw[thick, black, rounded corners=0.5pt] ($(NW_target) - (0.35cm, 0.35cm)$) rectangle ($(NW_target) + (0.35cm, 0.35cm)$);

 
  \coordinate (NW_ZoomPos) at ($(SE_target) + (4.5cm, 4.5cm)$); 
  \draw[thick, black, fill=white, rounded corners=2pt] ($(NW_ZoomPos) - (2.5cm, 2.5cm)$) rectangle ($(NW_ZoomPos) + (2.5cm, 2.5cm)$);
  
  \begin{scope}[shift={(NW_ZoomPos)}]
    \clip[rounded corners=2pt] (-2.5cm, -2.5cm) rectangle (2.5cm, 2.5cm);
    
    \begin{pgfinterruptboundingbox}
      \begin{scope}[x={(0.9cm,-0.076cm)}, y={(-0.06cm,0.95cm)}, z={(-0.44cm,-0.29cm)}, scale=350]
        
        \begin{scope}[shift={(-0.003, -0.999, -1.000)}] 
          
          \loadCoordinates
          
          \coordinate (v23) at (0.00505051, 1.0015, 0.9985);

          \drawInnerEdges{e_inner_z}
          \drawOuterEdges{e_outer_z}
          
          \foreach \i in {4, 5, 15, 16, 23} { \node at (v\i) [v_norm_z] {\i}; }
          
        \end{scope}
      \end{scope}
    \end{pgfinterruptboundingbox}
  \end{scope}

  
  \coordinate (SE_ZoomPos) at ($(SE_target) + (4.5cm, -1.0cm)$); 
  \draw[thick, black, fill=white, rounded corners=2pt] ($(SE_ZoomPos) - (2.5cm, 2.5cm)$) rectangle ($(SE_ZoomPos) + (2.5cm, 2.5cm)$);
  
  \begin{scope}[shift={(SE_ZoomPos)}]
    \clip[rounded corners=2pt] (-2.5cm, -2.5cm) rectangle (2.5cm, 2.5cm);
    
    \begin{pgfinterruptboundingbox}
      \begin{scope}[x={(0.9cm,-0.076cm)}, y={(-0.06cm,0.95cm)}, z={(-0.44cm,-0.29cm)}, scale=350]
        
        \begin{scope}[shift={(-0.958, -0.312, -0.312)}] 
          
          \loadCoordinates
          \drawInnerEdges{e_inner_z}
          \drawOuterEdges{e_outer_z}

          \foreach \i in {3, 6, 7, 8} { \node at (v\i) [v_norm_z] {\i}; }
          
        \end{scope}
      \end{scope}
    \end{pgfinterruptboundingbox}
  \end{scope}

  
  \draw[thick, dashed, gray, opacity=0.8] ($(NW_target) + (0.35cm, 0.35cm)$) -- ($(NW_ZoomPos) + (-2.5cm, 2.5cm)$);
  \draw[thick, dashed, gray, opacity=0.8] ($(NW_target) + (0.35cm, -0.35cm)$) -- ($(NW_ZoomPos) + (-2.5cm, -2.5cm)$);

  \draw[thick, dashed, gray, opacity=0.8] ($(SE_target) + (0.35cm, 0.35cm)$) -- ($(SE_ZoomPos) + (-2.5cm, 2.5cm)$);
  \draw[thick, dashed, gray, opacity=0.8] ($(SE_target) + (0.35cm, -0.35cm)$) -- ($(SE_ZoomPos) + (-2.5cm, -2.5cm)$);

\end{tikzpicture}

    \caption{Schlegel diagram of $P_R$ for the $\{2,2,3,1\}$ instance on the outer facet $x_1=0$ (facet edges in black).}
    \label{fig:p2_full}
    \vfill
\end{figure*}
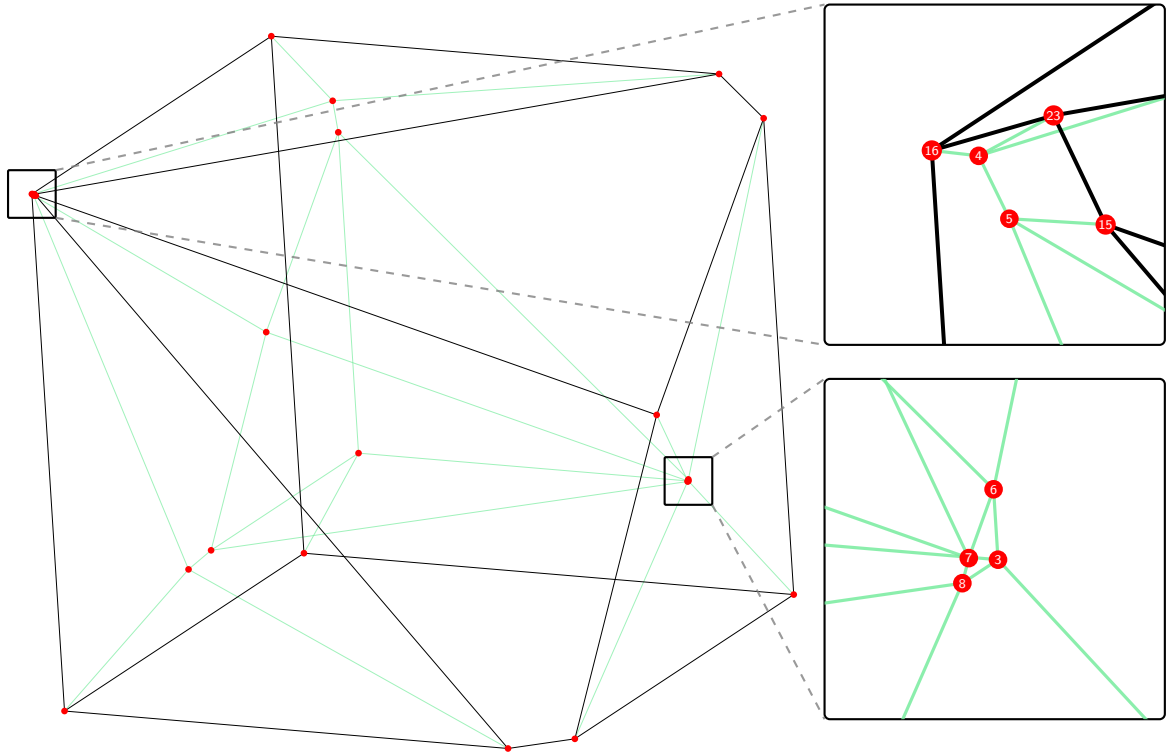


\subsection{Conditions for degeneracy in $P_R$}
In this section, we briefly elaborate conditions for $P_R$ to have degenerate vertices. 
The first, of course, is that $s_i = \frac{s_{\max}}{2}$ for some $i \in [2m]$.
We note that the example provided is for a polytope with an exact partition, and that the degenerate vertex contains $m$ ones, corresponding to this exact partition. In fact, this turns out to be required for such a degenerate vertex, as we will see.

Recall that vertices in $V_0$ and $V_2$ cannot be degenerate (as justified in the proof of Claim \ref{nondegLPre}), so we focus our attention on vertices in  $V_1$, i.e.~those $\ve{v}$ with $|\{v_i \in \{0,1\}\}| = 2m-1$. Note that for any vertex $\ve{v}$ of $P_R$ in $V_1$ and $V_2$, $|I_1(\ve{v})|\geq m-1$ (Lemma 8, \cite{NS26}).

\begin{lemma}
    A vertex $\ve{v}\in V_1$ in $P_R$ with $|I_1(\ve{v})| = m-1$ cannot be degenerate. 
\end{lemma}

\begin{proof}

Assume that a vertex $\ve{v} \in V_1$ with $|I_1(\ve{v})| = m-1$ is degenerate.  Then $\ve{v}$  
lies on $2m+1$ active constraints, and has one remaining fractional index.  Without loss of generality, this is the first co-ordinate, i.e.~$0< v_1<1$. 
When there is no element $s_i = \frac{s_{\max}}{2}$, then we are in the situation of Lemma~\ref{new}, and no vertices are degenerate. Further, if $i \ne 1$, Frieze and Teng's analysis of Claim~\ref{nondegLPre} gives that $\ve{v}$ does not lie on both (K1) and (K2).  

So we consider the situation when $s_1 = \frac{s_{\max}}{2}$. 
If both (K1) and (K2) are active, we have: 
\begin{align*}
    \sum_{i=1}^{2m} (M + s_i) x_i &= \frac{1}{2}S + mM + \epsilon \tag{K1}\\ 
\sum_{i=1}^{2m} (M + d_i)x_i &= \frac{1}{2} \sum_{i=1}^{2m} d_i + mM + \epsilon \tag{K2}
\end{align*}
We define a parameter $\alpha_{\ve{v}}$ which denotes the difference between the sum of variables in the support of a vertex $\ve{v}$ and the target sum of the input instance of \textsc{Exact Partition}. Symbolically, we have at a vertex $\ve{v}$, $$\displaystyle \alpha_{\ve{v}} = \frac{S}{2} - \sum_{i \in I_1(\ve{v})}s_i$$
From (K1), we obtain the following:
        \begin{align*}
            \sum_{i \in I_1(\ve{v})}(M+s_i) + (M+s_1)v_1 &= \frac{S}{2} + mM + \epsilon \\
            \sum_{i \in I_1(\ve{v})}s_i + (M+s_1)v_1 &= \frac{S}{2} + M + \epsilon \\
            v_1 &= \frac{\alpha_{\ve{v}} + M + \epsilon}{M+s_1} \text{\tag{1}}
        \end{align*}

Recall that $d_i = s_{\max}-s_i$ and $\sum d_i = 2m \cdot s_{\max} - S$. We have the following from (K2):
        \begin{align*}
           \sum_{i \in I_1(\ve{v})}(M+d_i) + (M+d_1)v_1 &= \frac{1}{2}\sum_{i=1}^{2m} d_i + mM + \epsilon \\
           (M+s_{\max}-s_1)v_1 &= s_{\max} + M - \alpha_{\ve{v}} + \epsilon\\
           v_1 &= \frac{s_{\max} + M - \alpha_{\ve{v}} + \epsilon}{M + s_{\max} -s_1}
        \end{align*}
Since $s_1 = \frac{s_{\max}}{2}$, $M+s_{\max}-s_1 = M + \frac{s_{\max}}{2} = M+s_1$, and from (K2), this becomes
\begin{align*}
    v_1 = \frac{s_{\max} + M - \alpha_{\ve{v}} + \epsilon}{M + s_1}\tag{2}
\end{align*}

Equating the two expressions for $v_1$, which now share a common denominator, we get 
$$ \alpha_v + M + \epsilon = s_{\max} + M - \alpha_v +\epsilon$$
Solving for $\alpha_v$ gives $\alpha_v = \displaystyle \frac{s_{\max}}{2} = s_1$.  But since $\epsilon>0$, we have $v_1 = \displaystyle \frac{M + s_1 + \epsilon}{M+s_1} > 1$, violating the box constraint.  The statement follows.

\end{proof}

\subsubsection{Existence of degenerate vertices}
We now show that whenever there is an exact partition and an element $s_i = \frac{s_{\max}}{2}$ for some $i \in [2m]$, that $P_R$ admits a degenerate vertex.\footnote{Each valid $m$-element subset yields as many distinct degenerate vertices in the exact partition polytope as there are elements with $s_i = s_{\max} /2$ in its complement.} 
First, we observe that at a vertex $\ve{v} \in P_R$ where $|I_1(\ve{v})|=m$, the elements at indices $I_1(\ve{v})$ necessarily encode an exact partition. This is implicit from Lemma 4 in~\cite{FT94}.  We elaborate it below following the additional details in~\cite{KN25}. 

\begin{lemma}[\cite{KN25}]\label{epv}
For a vertex $\ve{v}\in V_0$, if $|I_1(\ve{v})|=m$, then the elements at indices $I_1(\ve{v})$ represent an exact partition in the input instance.
\end{lemma} 

\begin{proof}

Consider $\alpha_{\ve{v}} = \frac{S}{2} - \sum_{i \in I_1(\ve{v})}s_i$ as before. The statement is equivalent to $\alpha_{\ve{v}} = 0$. 
We can see that we can't have $\alpha_v < 0$, as the constraint (K1) would be violated.

So assume $\alpha_v >0$.
Recall that $d_i = s_{\max}-s_i$ by definition in ILP2. In (K2), we have the following:
        \begin{align*}
            \sum_{i \in I_1(\ve{v})} (M+d_i)v_i &\leq 
               \frac{\sum_{i=1}^{2m} d_i}{2} + mM + \epsilon\\
            m(M + s_{\max}) - \sum_{i \in I_1(\ve{v})}s_i &\leq m(M+s_{\max}) - \frac{S}{2} + \epsilon \\
            0 &\leq -\frac{S}{2} + \sum_{i \in I_1(\ve{v})}s_i + \epsilon\\
            0 &\leq (-\alpha_{\ve{v}}) + \epsilon
         \end{align*}
     Hence $\alpha_v \le \epsilon$, which in turn is less than 0.5 by construction.     But this is impossible since $\alpha_v$ is positive and at least 0.5.
    
    Therefore, $\alpha_{\ve{v}} = 0$ and the statement follows. 
    
\end{proof}

Let us now suppose that there is an exact partition in the input instance and that the elements at indices $\{1,\dots,m\}$ are a valid subset. Let vertex $\ve{v}$ be a corresponding vertex in $V_0$ whose support consists of variables at indices $\{1,\dots,m\}$. 
\begin{observation}\label{obs}
    Consider a vertex $\ve{v} \in V_0$ where $|I_1(\ve{v})|=m$. Then, at $\ve{v}$, the slack in the RHS of (K1) and (K2) is $\epsilon$.
\end{observation}
\begin{proof}
By Lemma \ref{epv}, we have at $\ve{v}$:
\begin{align*}
    \sum_{i \in I_1(\ve{v})}(M+s_i)v_i = mM +\frac{S}{2} \tag{1}
\end{align*}
Therefore, in (K1)
\begin{align*}
    \sum_{i \in I_1(\ve{v})}(M+s_i)v_i &\leq  \frac{S}{2} + mM + \epsilon\\
    \sum_{i \in I_1(\ve{v})}(M+s_i)v_i - \frac{S}{2} - mM &\leq \epsilon\\
    0 &< \epsilon. 
\end{align*}
Similarly, we have:
\begin{equation*}
    \sum_{i \in I_1(\ve{v})}(M+d_i)v_i = mM + m\cdot s_{\max} - \frac{S}{2} \tag{2}
\end{equation*}
Therefore, in (K2), we have
\begin{align*}
    \sum_{i \in I_1(\ve{v})}(M+d_i)v_i &\leq mM+ \frac{\sum d_i}{2}+ \epsilon\\
    \sum_{i \in I_1(\ve{v})}(M+d_i)v_i &\leq mM + m\cdot s_{max} - \frac{\sum s_i}{2}+\epsilon\\ 
    0 &< \epsilon. 
\end{align*}

\end{proof}

Recall a vertex $\ve{v} \in V_1$ (where $|\{v_i \in \{0,1\}\}| = 2m-1$) is degenerate when both (K1) and (K2) are active. Consider $\ve{v^*} \in V_1$ where $|I_1(\ve{v^*})| = m$, indices $\{1,\dots,m\}$ form a valid subset, and there is an index $i \in I_f(\ve{v^*})$. If at least one $i \in \{m+1,\dots, 2m\}$ satisfies $s_i = \frac{s_{\max}}{2}$, the coefficients of $v^*_i$ in (K1) and (K2) are equal since $M+s_i = M + d_i = M+s_{max} - s_i$. We show $\ve{v^*}$ is degenerate.

\begin{theorem}
    Vertex $\ve{v^*} \in V_1$ (as defined above, with $s_i = \frac{s_{\max}}{2}$ for some $i \in \{m+1,\ldots,2m\}$) is degenerate.
\end{theorem}

\begin{proof}
    Recall that $I^*(\ve{v}) = I_1(\ve{v}) \cup I_f(\ve{v})$. Since $\ve{v^*}\in V_1$, let constraint (K1) be tight at $\ve{v^*}$. At the index $i \in I_f(\ve{v^*})$, $v^*_i$ is as follows:
    \begin{align*}
    \sum_{k \in I^*(\ve{v^*})}(M+s_k)v^*_k &=  \frac{S}{2} + mM + \epsilon    
    \end{align*}
    \begin{align*}
        (M+s_i)v^*_i &= \epsilon \quad (\because \text{by Observation \ref{obs}}) \\
        v^*_i &= \frac{\epsilon}{(M+\frac{s_{\max}}{2})}\\
        v^*_i &= \frac{2\epsilon}{(2M+s_{\max})} \quad (\because s_i = \frac{s_{\max}}{2})
    \end{align*}
    Evaluating the LHS of (K2) at this point, we have:
    \begin{align*}
        \sum_{k \in I^*(\ve{v^*})}(M+d_k)v^*_k & = mM+ \frac{\sum d_i}{2}+ (M+d_i)v^*_i\\
        & \quad \quad (\because \text{Lemma \ref{epv}})
    \end{align*}
    Given that the slack in the RHS of (K2) is $\epsilon$ (per Observation \ref{obs}), the value of $v^*_i$ required to bridge this slack is:
    \begin{align*}
        (M+d_i)v^*_i &= \epsilon \\
        v^*_i &= \frac{\epsilon}{(M+s_{\max} -s_i)} \quad (\because d_i = s_{\max} - s_i)\\
        v^*_i &= \frac{2\epsilon}{(2M+s_{\max})} \quad (\because s_i = \frac{s_{\max}}{2})
    \end{align*}
    
    Since the obtained values of $v^*_i$ from (K1) and (K2) are equal, both constraints are tight at vertex $\ve{v^*}$ apart from $2m-1$ box constraints. The statement follows. 
\end{proof}

\section*{Acknowledgments}
This research was partially funded by NSERC Discovery Grant \texttt{RGPIN-2024-06800} to TS. The authors acknowledge using Polymake to generate the tikz code and Google Gemini to adapt it for the illustrations in this document.

\bibliographystyle{elsarticle-num} 
\bibliography{cas-refs}

\end{document}